\documentclass[11pt]{amsart}
\usepackage{amsmath,amsthm,amssymb,mathtools,enumitem,hyperref}

\newtheorem{theorem}{Theorem}[section]
\newtheorem{lemma}[theorem]{Lemma}
\newtheorem{prop}[theorem]{Proposition}
\newtheorem{cor}[theorem]{Corollary}
\theoremstyle{definition}
\newtheorem{definition}[theorem]{Definition}

\newtheorem{question}[theorem]{Question}

\newtheorem{remark}[theorem]{Remark}

\newcommand{\Aut}{\operatorname{Aut}}

\title{On minimal non-sofic and $\omega$-non-sofic groups}
\author{Kıvanç ERSOY}
\address{Institut für Mathematik, Freie Universität Berlin, Arnimallee 7, 14195 Berlin, Germany}
\email{ersoy@zedat.fu-berlin.de}\date{}

\begin{document}

\maketitle

\begin{abstract}
We investigate structural properties of non-sofic groups, assuming that such groups exist. We introduce and study two classes: minimal non-sofic groups and $\omega$-non-sofic groups. For minimal non-sofic groups, we establish strong structural restrictions. In particular, we show that if $G$ is a minimal non-sofic group and $M$ is a finitely generated residually finite maximal normal subgroup of $G$, then $M$ is central and $G$ is a perfect central extension of a finitely generated non-amenable simple group.

On the other hand, we show that locally graded non-sofic groups are necessarily $\omega$-non-sofic. More precisely, such groups contain finitely generated non-sofic subgroups admitting strictly decreasing chains of finitely generated normal subgroups whose intersection is non-trivial and lies in the profinite residual.

Finally, using results on existentially closed groups, we prove that the existence of a non-sofic group implies the existence of a countable existentially closed non-sofic group whose centralizers form a densely ordered chain of non-sofic subgroups of order type $(\mathbb{Q},\leq)$. In particular, we show that if a non-sofic group exists, then the class of $\omega$-non-sofic groups is non-empty. Moreover, we prove that the existence of a non-sofic group implies the existence of a non-sofic group of unbounded exponent.
\end{abstract}

\section{Introduction}

Sofic groups were introduced by Gromov \cite{Gromov1999} as a common generalization of residually finite and amenable groups. For background, definitions, and notation, we refer to \cite{Gromov1999} and the expository articles \cite{Pestov2008, Weiss2000}.

\begin{definition}\cite[Definition 3.1]{Pestov2008}
A group $G$ is called \emph{sofic} if it is isomorphic to a subgroup of a metric ultrapower of a suitable family of finite symmetric groups equipped with the normalized Hamming metric.

Equivalently, there exist
\begin{itemize}
\item a set $A$,
\item an ultrafilter $\mathcal{U}$ on $A$, and
\item a map $\alpha \mapsto n(\alpha)$ from $A$ to $\mathbb{N}$,
\end{itemize}
such that
\[
G \leq \prod\nolimits_{\alpha \in A} \bigl(S_{n(\alpha)}, d_{\mathrm{hamm}}\bigr)\Big/ \mathcal{U}.
\]
\end{definition}

There are several equivalent definitions of sofic groups, including approximation by finite symmetric groups, embeddings into metric ultraproducts, and approximation of Cayley graphs. We refer to \cite{Weiss2000, elsza, Pestov2008, CapraroLupini2015} for the equivalence of these formulations.

As observed in \cite[Examples 4.1, 4.2, and 4.4]{Pestov2008}, finite groups, residually finite groups, and amenable groups are sofic. It remains open whether every group is sofic.

\begin{question}\label{tem}
\cite[Question 3.8]{Pestov2008} Is every (countable discrete) group sofic?
\end{question}

In this paper, we study the following two classes of groups:

\begin{definition}
A group $G$ is called \emph{minimal non-sofic} if $G$ is not sofic, but every proper subgroup of $G$ is sofic.
\end{definition}

\begin{definition}
A non-sofic group $G$ is called \emph{$\omega$-non-sofic} if it admits a strictly descending chain of non-sofic subgroups.
\end{definition}

Since it is not known whether non-sofic groups exist at all, the class of minimal non-sofic groups may be empty. Nevertheless, studying such groups provides a natural approach to Question~\ref{tem}. Well-known candidates for non-soficity include Higman's group, Thompson's groups, and Tarski monsters; they are infinite, possibly non-amenable, and not residually finite. In particular, if a Tarski monster were non-sofic, then it would be minimal non-sofic, since all of its proper subgroups are finite. By contrast, Thompson's groups would be $\omega$-non-sofic if they were non-sofic, since they contain proper subgroups isomorphic to themselves.

This paper is intended as a first step in a broader project aimed at understanding the possible structure of non-sofic groups, assuming that such groups exist. The general approach is to combine classical results from group theory with the study of two natural classes of candidate non-sofic groups, namely minimal non-sofic groups and $\omega$-non-sofic groups.

In this work, we investigate structural properties of non-sofic groups under this perspective. It is straightforward to observe that if a non-sofic group exists, then it either contains a minimal non-sofic subgroup or is $\omega$-non-sofic. Our aim is to study these two classes in the hope that the resulting structural constraints may eventually lead to contradictions, thereby contributing to the problem of whether certain specific groups are sofic.

We first restrict our attention to periodic minimal non-sofic groups. Since natural candidates include Tarski monsters, it is reasonable to ask the following questions.

\begin{question}\label{pexp}
Let $G$ be a periodic minimal non-sofic group. Must $G$ have prime exponent?
\end{question}

\begin{question}\label{bexp}
Let $G$ be a periodic minimal non-sofic group. Must $G$ have bounded exponent?
\end{question}

Although we do not know the answers to Questions~\ref{pexp} and~\ref{bexp} for minimal non-sofic groups, we record the following consequence of a result of Ivanov; see \cite{ivols}.

\begin{prop}\label{expthm}
If there exists a periodic non-sofic group without involutions, then there exists a divisible periodic non-sofic group.
\end{prop}

\begin{cor}
If there exists a periodic non-sofic group without involutions, then there exists a non-sofic group of unbounded exponent.
\end{cor}

Locally graded groups form a broad class containing all residually finite groups. It is unknown whether every locally graded group is sofic.

\begin{question}\label{locgr}
Is every locally graded group sofic?
\end{question}

We prove the following result, which is related to Question~\ref{locgr}.

\begin{theorem}\label{locgrnon}
Let $G$ be a locally graded non-sofic group. Then $G$ contains a finitely generated non-sofic subgroup $H$ and a strictly decreasing sequence of finitely generated normal subgroups
\[
H=N_0 > N_1 > N_2 > \cdots
\]
of $H$ such that
\begin{itemize}
\item each $N_i$ is non-sofic;
\item the intersection $R=\bigcap_{i=0}^\infty N_i$ is non-trivial and is contained in the profinite residual of $H$;
\item $H/R$ is residually finite.
\end{itemize}
In particular, $G$ is $\omega$-non-sofic.
\end{theorem}

\begin{remark}
One may observe that minimal non-sofic groups cannot be locally graded. Namely, if $G$ is a locally graded minimal non-sofic group, since $G$ is finitely generated, it admits a finitely generated normal subgroup $N$ of finite index. By minimality, $N$ is sofic, and the quotient $G/N$ is finite. This contradicts the assumption that $G$ is non-sofic.
\end{remark}

Finally, we prove the following main result.

\begin{theorem}\label{resfinnormax}
Let $G$ be a minimal non-sofic group and let $M$ be a maximal normal subgroup. If $M$ is residually finite and finitely generated, then $M$ is central. In this case, $G$ is a perfect central extension of a finitely generated non-amenable simple group by $M$.
\end{theorem}

\section{Preliminaries}

We follow the definitions and notation of \cite[Chapter 2]{CapraroLupini2015}. The following proposition records several closure properties of the class of sofic groups.

\begin{prop}\cite[Proposition 2.4.1]{CapraroLupini2015}\label{temel}
     The class of sofic groups is closed under the following operations:
     \begin{enumerate}
         \item subgroups;
         \item direct limits;
         \item direct products;
         \item inverse limits;
         \item extensions by amenable groups;
         \item free products;
         \item free products amalgamated over an amenable group;
         \item HNN extensions over an amenable group;
         \item graph products.
     \end{enumerate}
\end{prop}

The following result gives an alternative definition of soficity, which is sometimes easier to work with.

\begin{theorem}\label{soficapprox}\cite[Theorem 3.5]{Pestov2008}
A group $G$ is sofic if and only if for every finite subset $F \subseteq G$ and every $\varepsilon>0$, there exist a natural number $n$ and a map
\[
\theta:F\to S_n
\]
such that
\begin{enumerate}
    \item if $g,h,gh\in F$, then
    \[
    d_{\mathrm{hamm}}(\theta(g)\theta(h),\theta(gh))<\varepsilon;
    \]
    \item if $e\in F$, then
    \[
    d_{\mathrm{hamm}}(\theta(e),\mathrm{Id})<\varepsilon;
    \]
    \item for all distinct $x,y\in F$,
    \[
    d_{\mathrm{hamm}}(\theta(x),\theta(y))\geq \frac14.
    \]
\end{enumerate}
\end{theorem}

The following proposition allows us to restrict attention to countable groups.

\begin{prop}\label{loc}
If $G$ is a non-sofic group, then $G$ contains a countable non-sofic subgroup.
\end{prop}

\begin{proof}
Suppose that every countable subgroup of $G$ is sofic. Let $F \subseteq G$ be finite and let $\varepsilon>0$. Then the subgroup $\langle F\rangle$ generated by $F$ is countable, hence sofic. By Theorem~\ref{soficapprox}, there exists an $(F,\varepsilon)$-approximation for $\langle F\rangle$, and therefore for $G$ on the set $F$. Since $F$ and $\varepsilon$ were arbitrary, Theorem~\ref{soficapprox} implies that $G$ is sofic. This contradiction shows that $G$ contains a countable non-sofic subgroup.
\end{proof}

In particular, Proposition \ref{loc} shows that soficity is a local property. For that reason, we may restrict our attention to countable groups. Indeed, by \cite[Theorem 2]{elsza2} or by \cite[Proposition 2.2.5]{CapraroLupini2015}, every countable discrete sofic group is hyperlinear. One may also discuss \textit{minimal non-hyperlinear groups}, but we will not do that in the present paper.
\begin{prop}\label{fgperfect}
Let $G$ be a minimal non-sofic group. Then the following hold:
\begin{enumerate}
    \item $G$ is finitely generated. In particular, $G$ has maximal subgroups.
    \item $G$ has no non-trivial amenable quotients. In particular, $G$ is perfect.
    \item $G$ has maximal normal subgroups.
\end{enumerate}
\end{prop}

\begin{proof}
\begin{enumerate}
    \item By Proposition~\ref{loc}, the group $G$ is countable. Suppose that $G$ is not finitely generated. Write
    \[
    G=\{g_i : i \in \mathbb{N}\},
    \]
    and let
    \[
    G_i=\langle g_1,\dots,g_i\rangle.
    \]
    Then each $G_i$ is finitely generated and proper, hence sofic by minimality. Since $G$ is the direct limit of the groups $G_i$, Proposition~\ref{temel} implies that $G$ is sofic, a contradiction. Therefore, $G$ is finitely generated.

    By \cite[(5) Theorem]{neu37}, every proper subgroup of a finitely generated group is contained in a maximal subgroup.
    
    \item Let $N \trianglelefteq G$ be such that $G/N$ is amenable. If $N$ is proper, then $N$ is sofic by minimality, and therefore $G$ is sofic, since an extension of a sofic group by an amenable group is sofic. This contradiction shows that $N=G$. Hence $G$ has no non-trivial amenable quotients.

    Since the abelianization $G/[G,G]$ is amenable, it follows that $G/[G,G]=1$. Therefore $G=[G,G]$, and so $G$ is perfect.

    \item  Let
\[
\mathcal{N}=\{N\trianglelefteq G : N\neq G\},
\]
partially ordered by inclusion. Since $\{1\}\in\mathcal{N}$, the set $\mathcal{N}$ is non-empty. Let $\mathcal{C}=\{N_i\}_{i\in I}$ be a chain in $\mathcal{N}$, and set
\[
N=\bigcup_{i\in I} N_i.
\]
Then $N$ is a subgroup of $G$, and since each $N_i$ is normal and the family is linearly ordered by inclusion, $N$ is normal in $G$. 

If $N=G$, then since $G$ is finitely generated, all generators of $G$ lie in some $N_i$, which implies $N_i=G$, a contradiction. Hence $N\neq G$, so $N\in\mathcal{N}$. Therefore every chain in $\mathcal{N}$ has an upper bound, and by Zorn's lemma, $G$ has a maximal normal subgroup.
\end{enumerate}
\end{proof}

Additional structure can be derived for maximal normal subgroups.

\begin{prop}
Let $G$ be a minimal non-sofic group. Then $G$ has a maximal normal subgroup $M$ such that $G/M$ is a finitely generated non-amenable simple group.
\end{prop} 

\begin{proof}
Since $M$ is maximal, the quotient $G/M$ is simple. By Proposition~\ref{temel}, it is non-amenable. Since $G$ is finitely generated, by \cite[1.4 A]{robb1}, every quotient of $G$ is finitely generated.
\end{proof}

\begin{prop}\label{conjinf}
Let $G$ be a minimal non-sofic group. Then, for every non-central element $x$, the conjugacy class $x^G$ is infinite.
\end{prop}

\begin{proof}
Assume that some element has only finitely many conjugates. Then $C_G(x)$ has finite index, hence $G$ has a normal subgroup of finite index by \cite[Theorem 3.14]{rotman}, contradicting Proposition~\ref{temel}.
\end{proof}

The following proposition is immediate.

\begin{prop}\label{finin}
Every finite normal subgroup of a minimal non-sofic group is central.
\end{prop}

\begin{proof}
Let $G$ be a minimal non-sofic group and let $N$ be a finite normal subgroup. Clearly, $C_G(N)$ is also normal and $G/C_G(N)$ embeds in $\Aut(N)$, which is finite. This contradicts the fact that $G$ has no amenable images. Therefore, $C_G(N)=G$.
\end{proof}

\begin{remark}
For every minimal non-sofic group $G$ and every maximal normal subgroup $M$, the quotient $G/M$ is non-amenable and simple. Since $G$ is finitely generated, $G/M$ is finitely generated.
\end{remark}





\section{Existentially closed non-sofic groups and $\omega$-non-soficity}

In this section, we show that the existence of a non-sofic group leads naturally to the existence of $\omega$-non-sofic groups with a rich internal structure. More precisely, we prove that, if a non-sofic group exists, then there is a countable existentially closed group containing it whose centralizers form a densely ordered chain of non-sofic subgroups of order type $(\mathbb{Q},\leq)$.

\begin{prop}\label{ecns}
Assume that there exists a non-sofic group. Then there exists a countable $\aleph_0$-existentially closed non-sofic group $G$.
\end{prop}

\begin{proof}
This follows from standard embedding results for existentially closed groups; see, for instance, \cite{bek}.
\end{proof}

In this section from now on, let $G$ be a countable $\aleph_0$-existentially closed non-sofic group.

\begin{prop}\label{ecomega}
The group $G$ is $\omega$-non-sofic.
\end{prop}

\begin{proof}
By \cite[Theorem C]{bek}, there exists a family $\{A_q\}_{q\in\mathbb{Q}}$ of finitely generated non-abelian free subgroups of $G$ such that the family
\[
\{C_G(A_q)\}_{q\in\mathbb{Q}}
\]
is linearly ordered by inclusion with order type $(\mathbb{Q},\leq)$. Moreover, each $C_G(A_q)$ is isomorphic to $G$ by \cite[Theorem C]{bek}. Since $G$ is non-sofic, it follows that each $C_G(A_q)$ is non-sofic.

Now choose a strictly descending sequence in $(\mathbb{Q},\leq)$. The corresponding sequence of centralizers is then a strictly descending chain of non-sofic subgroups of $G$. Therefore, $G$ is $\omega$-non-sofic.
\end{proof}

\begin{prop}\label{centralizersstructure}
Let $A \leq G$ be a finitely generated subgroup. Then $C_G(A)$ is non-sofic. Moreover, if $Z(A)=1$, then
\[
C_G(A)\cong G.
\]
In particular, in this case $C_G(A)$ is a countable $\aleph_0$-existentially closed non-sofic group.
\end{prop}

\begin{proof}
By \cite[Corollary 3.3]{bek}, the group $G$ embeds into $C_G(A)$. Since $G$ is non-sofic and soficity is closed under taking subgroups, it follows that $C_G(A)$ is non-sofic.

If, in addition, $Z(A)=1$, then $C_G(A)\cong G$ by \cite[Theorem B]{bek}. The final statement follows immediately.
\end{proof}

\begin{prop}\label{simplequotients}
Let $A \leq G$ be a finitely generated subgroup. Then the quotient $C_G(A)/Z(A)$ is simple. Moreover, if $Z(A)=1$, then
\[
C_G(A)/Z(A)\cong G.
\]
\end{prop}

\begin{proof}
The simplicity follows from \cite[Proposition 3.7]{bek}, while the second statement follows from \cite[Theorem B]{bek}.
\end{proof}

The key consequence of this section is the following.

\begin{theorem}
If there exists a non-sofic group, then the class of $\omega$-non-sofic groups is non-empty.
\end{theorem}

\begin{proof}
Proposition~\ref{ecomega} shows that, under the existence of non-sofic groups, the class of $\omega$-non-sofic groups is non-empty. In fact, it contains countable existentially closed groups whose centralizers form a densely ordered chain of non-sofic subgroups of order type $(\mathbb{Q},\leq)$, each of them isomorphic to the given non-sofic group.
\end{proof}

\section{Main results on minimal non-sofic groups}

\begin{lemma}
Let \(G\) be a minimal non-sofic group, and let \(H \lhd G\) be a normal subgroup. If \(\operatorname{Out}(H)\) is amenable, then
\[
G = H C_G(H).
\]
\end{lemma}

\begin{proof}
Consider the conjugation homomorphism
\[
\varphi : G \to \operatorname{Aut}(H),
\]
and let
\[
\bar{\varphi} : G \to \operatorname{Out}(H)
\]
be its composition with the natural quotient map \(\operatorname{Aut}(H) \to \operatorname{Out}(H)\).

Since \(\operatorname{Out}(H)\) is amenable, the image of \(\bar{\varphi}\) is amenable. As \(G\) is minimal non-sofic, it admits no non-trivial amenable quotients. Hence \(\bar{\varphi}\) is trivial.

It follows that for every \(g \in G\), the automorphism of \(H\) induced by conjugation by \(g\) is inner. Thus there exists \(h_g \in H\) such that
\[
g x g^{-1} = h_g x h_g^{-1} \qquad \text{for all } x \in H.
\]
Therefore \(h_g^{-1} g \in C_G(H)\), and hence \(g \in H C_G(H)\). This proves that \(G = H C_G(H)\).
\end{proof}

\begin{theorem}
Let \(G\) be a minimal non-sofic group. Then \(G\) has no infinite simple proper normal subgroup \(N\) such that \(\operatorname{Out}(N)\) is amenable.
\end{theorem}

\begin{proof}
Assume that \(H \lhd G\) satisfies the above conditions. By the lemma,
\[
G = H C_G(H).
\]
Since \(H\) is an infinite simple group, \(Z(H)=1\), and hence
\[
H \cap C_G(H) = 1.
\]
Thus
\[
G \cong H \times C_G(H).
\]

The subgroup \(C_G(H)\) is normal in \(G\). It cannot be equal to \(G\), since otherwise \(H \subseteq Z(G)\), contradicting the fact that \(H\) is non-abelian. Hence \(C_G(H)\) is a proper normal subgroup of \(G\), and therefore sofic.

By minimality of $G$, one has \(H\)  sofic. It follows that \(G\) is sofic, as a direct product of sofic groups, which leads to  a contradiction.
\end{proof}
\begin{lemma}\label{locsimp} Let $G$ be a simple group of Lie type over a locally finite field. Then $\operatorname{Out} G$ is amenable.
\end{lemma}
\begin{proof}
By \cite[Theorem 30]{ste}, the outer automorphism group is a product of the groups of diagonal automorphisms, field automorphisms and graph automorphisms. The first and last are cyclic where the second is procyclic, hence the product is solvable. Therefore, it is amenable.\end{proof}
\begin{cor}
Let \(G\) be a minimal non-sofic group. Then \(G\) has no normal subgroup isomorphic to an infinite simple locally finite linear group.
\end{cor}

\begin{proof}
Let \(H \lhd G\) be such a subgroup. Then \(H\) is locally finite and hence  it is sofic. Moreover, \(\operatorname{Out}(H)\) is amenable by Lemma \ref{locsimp}. The result follows from the previous theorem.
\end{proof}

Theorem \ref{resfinnormax} is the main theorem of this paper, which gives a very restricted structure for minimal non-sofic groups with finitely generated residually finite maximal normal subgroup. 
\begin{proof}[Proof of Theorem \ref{resfinnormax}]
Let \(G\) be a minimal non-sofic group and let \(M\) be a maximal normal subgroup that is finitely generated and residually finite. Then, by \cite{baum}, \(\operatorname{Aut}(M)\) is residually finite, and therefore so is
\[
G / C_G(M).
\]
On the other hand, by Proposition~\ref{temel}, \(G\), and hence \(G / C_G(M)\), has no finite images. A residually finite group with no finite images is trivial, so \(C_G(M) = G\). Therefore, \(M\) is central.

Since \(G/M\) is a finitely generated simple group and \(M\) is central, it follows that \(G\) is a perfect central extension of a finitely generated simple group.
\end{proof}
Finally we  prove the result on locally graded groups:

\begin{proof}[Proof of Theorem~\ref{locgrnon}]
Since soficity is a local property, $G$ contains a finitely generated non-sofic subgroup $H$. Being a subgroup of a locally graded group, $H$ is itself locally graded.

Let
\[
K=\bigcap \{N \trianglelefteq H \mid [H:N]<\infty\}
\]
be the profinite residual of $H$. Since $H$ is finitely generated, it has only countably many normal subgroups of finite index. Enumerate them as
\[
M_1,M_2,M_3,\ldots
\]
and define
\[
L_n=\bigcap_{i=1}^n M_i \qquad (n\geq 1).
\]
Then each $L_n$ is a finite-index normal subgroup of $H$, and
\[
\bigcap_{n=1}^\infty L_n = K.
\]

We claim that each $L_n$ is non-sofic. Indeed, $H/L_n$ is finite, hence amenable. If $L_n$ were sofic, then $H$ would be sofic as an extension of a sofic group by an amenable group, a contradiction.

Next, we show that the sequence $(L_n)$ does not stabilize. Suppose that $L_n=L_{n+1}=L_{n+2}=\cdots$ for some $n$. Then $L_n=K$. Since $L_n$ is a finitely generated subgroup of the locally graded group $H$, the group $L_n$ is locally graded. As $L_n$ is non-sofic, it is non-trivial, so it has a proper subgroup of finite index. Hence it also has a proper normal subgroup of finite index. Such a subgroup is then a finite-index normal subgroup of $H$ contained in $K$, contradicting the definition of $K$ as the intersection of all finite-index normal subgroups of $H$. Therefore, $(L_n)$ does not stabilize.

Removing repetitions if necessary, we obtain a strictly decreasing sequence
\[
H=N_0 > N_1 > N_2 > \cdots
\]
of finitely generated normal subgroups of $H$, each of which is non-sofic, and such that
\[
\bigcap_{i=0}^\infty N_i = K.
\]
Set
\[
R=\bigcap_{i=0}^\infty N_i.
\]
Then $R=K$, so $R$ is non-trivial, since otherwise $H$ would be residually finite and hence sofic. Moreover, $R$ is contained in the profinite residual of $H$ (in fact, it coincides with it), and $H/R$ is residually finite.
\end{proof}
Proposition \ref{expthm} is not difficult to observe, but it is probably useful for further study, since it proves that if there is a periodic non-sofic group without involutions, then there exists a non-sofic group of unbounded exponent.
\begin{proof}[Proof of Proposition \ref{expthm}]
Let \(G\) be a periodic non-sofic group without involutions. By \cite[Theorem 20]{ivols}, \(G\) embeds into a divisible torsion group. The result follows.
\end{proof}

\bibliographystyle{amsplain}
\bibliography{ref}

@incollection{Gromov1999,
  author       = {Gromov, Mikhail},
  title        = {Endomorphisms of symbolic algebraic varieties},
  booktitle    = {Journal of the European Mathematical Society (JEMS)},
  volume       = {1},
  number       = {2},
  pages        = {109--197},
  year         = {1999},
  doi          = {10.1007/s100970050008},
  note         = {Introduced sofic groups}
}

@book{CapraroLupini2015,
  author    = {Valerio Capraro and Martino Lupini},
  title     = {Introduction to Sofic and Hyperlinear Groups and Connes' Embedding Conjecture},
  series    = {Lecture Notes in Mathematics},
  volume    = {2136},
  publisher = {Springer International Publishing, Cham},
  year      = {2015},
  isbn      = {978-3-319-19332-8},
  doi       = {10.1007/978-3-319-19333-5}
}

@article{neu37,
 author = {Neumann, B. H.},
 title = {Some remarks on infinite groups},
 fjournal = {Journal of the London Mathematical Society},
 journal = {J. Lond. Math. Soc.},
 issn = {0024-6107},
 volume = {12},
 pages = {120--127},
 year = {1937},
 doi = {10.1112/jlms/s1-12.46.120}
 }

@article{bek,
 author = {Brescia, Mattia and Ersoy, K{\i}van{\c{c}} and Kuzucuo{\u{g}}lu, Mahmut},
 title = {{{\(\kappa\)}}-existentially closed groups and centralizers},
 fjournal = {Journal of Algebra},
 journal = {J. Algebra},
 issn = {0021-8693},
 volume = {679},
 pages = {37--55},
 year = {2025},
 doi = {10.1016/j.jalgebra.2025.05.013}
}

@misc{ivols,
 author = {Ivanov, Sergej V. and Ol'shanskij, Aleksander Yu.},
 title = {Some applications of graded diagrams in combinatorial group theory},
 year = {1991},
 howpublished = {Groups, {Vol}. 2, {Proc}. {Int}. {Conf}., {St}. {Andrews}/{UK} 1989, {Lond}. {Math}. {Soc}. {Lect}. {Note} {Ser}. 160, 258-308 (1991).}
}

@book{rotman,
 author = {Rotman, Joseph J.},
 title = {An introduction to the theory of groups.},
 edition = {4th ed.},
 fseries = {Graduate Texts in Mathematics},
 series = {Grad. Texts Math.},
 issn = {0072-5285},
 volume = {148},
 isbn = {0-387-94285-8},
 year = {1995},
 publisher = {New York, NY: Springer-Verlag}
 }

@article{Pestov2008,
  author       = {Pestov, Vladimir},
  title        = {Hyperlinear and sofic groups: a brief guide},
  journal      = {Bulletin of Symbolic Logic},
  volume       = {14},
  number       = {4},
  pages        = {449--480},
  year         = {2008},
  publisher    = {Cambridge University Press},
  doi          = {10.2178/bsl/1231081374}
}

@article{elsza,
 author = {Elek, G{\'a}bor and Szab{\'o}, Endre},
 title = {On sofic groups.},
 fjournal = {Journal of Group Theory},
 journal = {J. Group Theory},
 issn = {1433-5883},
 volume = {9},
 number = {2},
 pages = {161--171},
 year = {2006},
 doi = {10.1515/JGT.2006.011},
}

@book{ste,
 author = {Steinberg, Robert},
 title = {Lectures on {Chevalley} groups},
 fseries = {University Lecture Series},
 series = {Univ. Lect. Ser.},
 issn = {1047-3998},
 volume = {66},
 isbn = {978-1-4704-3105-1; 978-1-4704-3631-5},
 year = {2016},
 publisher = {Providence, RI: American Mathematical Society (AMS)},
 doi = {10.1090/ulect/066}
}

@book{robb1,
 author = {Robinson, Derek J. S.},
 title = {Finiteness conditions and generalized soluble groups. {Part} 1.},
 fseries = {Ergebnisse der Mathematik und ihrer Grenzgebiete},
 series = {Ergeb. Math. Grenzgeb.},
 volume = {62},
 year = {1972},
 publisher = {Springer-Verlag, Berlin}
}

@article{elsza2,
 author = {Elek, G{\'a}bor and Szab{\'o}, Endre},
 title = {Hyperlinearity, essentially free actions and {{\(L^2\)}}-invariants. {The} sofic property},
 fjournal = {Mathematische Annalen},
 journal = {Math. Ann.},
 issn = {0025-5831},
 volume = {332},
 number = {2},
 pages = {421--441},
 year = {2005},
 doi = {10.1007/s00208-005-0640-8}
}

@article{Weiss2000,
  author       = {Weiss, Benjamin},
  title        = {Sofic groups and dynamical systems},
  journal      = {Sankhy\=a: The Indian Journal of Statistics, Series A},
  volume       = {62},
  number       = {3},
  pages        = {350--359},
  year         = {2000},
  note         = {One of the first expositions of sofic groups}
}

@article{baum,
 author = {Baumslag, G.},
 title = {Automorphism groups of residually finite groups},
 fjournal = {Journal of the London Mathematical Society},
 journal = {J. Lond. Math. Soc.},
 issn = {0024-6107},
 volume = {38},
 pages = {117--118},
 year = {1963},
 doi = {10.1112/jlms/s1-38.1.117}
}

\end{document}